\documentclass[11pt]{amsart}
\usepackage{amsmath} 
\usepackage{amsthm}
\usepackage{amssymb} 
\usepackage{amsfonts}
\usepackage{graphicx} 
\usepackage{array,makecell}
\usepackage{courier}
\usepackage{hyperref}
\usepackage[shortlabels]{enumitem}
\setlist[enumerate]{label=\normalfont{(\roman*)}}

\newtheorem{theorem}{Theorem}[section]

\newtheorem{proposition}[theorem]{Proposition}
\newtheorem{remark}[theorem]{Remark}

\newtheorem{corollary}[theorem]{Corollary}
\newtheorem*{easytheorem}{Theorem \ref{easytheorem}}
\newtheorem*{maintheorem}{Theorem \ref{maintheorem}}

\usepackage{accents}

\DeclareMathSymbol{\wtilde}{\mathord}{largesymbols}{"65}
\newcommand\lowerwtilde{%
  \text{\smash{\raisebox{-1.3ex}{%
    $\wtilde$}}}}
\DeclareRobustCommand\newtilde[1]{%
  \mathchoice
    {\accentset{\displaystyle\lowerwtilde}{#1}}
    {\accentset{\textstyle\lowerwtilde}{#1}}
    {\accentset{\scriptstyle\lowerwtilde}{#1}}
    {\accentset{\scriptscriptstyle\lowerwtilde}{#1}}
    }

\theoremstyle{definition}
\newtheorem{definition}[theorem]{Definition}

\newtheorem{construction}[theorem]{Construction}

\newtheorem*{notation}{Notation}

\newcommand{\Z}{\mathbb{Z}}

\let\int\relax
\newcommand{\int}{\mathring}

\newcommand{\pt}{\mathrm{pt}}

\usepackage[margin=1.5in]{geometry}
\newcommand{\boundary}{\partial}

      \usepackage{xcolor}

\title[A concordance analogue of the $4$-dimensional light bulb theorem]{A concordance analogue of the $4$-dimensional light bulb theorem}
\author[Maggie Miller]{Maggie Miller}
\address{Department of Mathematics, Princeton University, Princeton, NJ 08544}
\email[Maggie Miller]{maggiem@math.princeton.edu}
\urladdr{\url{https://web.math.princeton.edu/~maggiem/}}

\begin{document}
\maketitle
\begin{abstract}

We prove a concordance analogue of Gabai's $4$-dimensional light bulb theorem. That is, we show that when $R$ and $R'$ are homotopically embedded $2$-spheres in a $4$-manifold $X^4$ where $\pi_1(X^4)$ has no $2$-torsion and {\emph{one}} of $R$ or $R'$  has a transverse sphere, then $R$ and $R'$ are concordant. When $\pi_1(X^4)$ has $2$-torsion, we give a similar statement with extra hypotheses as in the $4$-dimensional light bulb theorem. We also give similar statements when $R$ and $R'$ are orientable positive-genus surfaces.

\end{abstract}
\setcounter{tocdepth}{2}
\setcounter{equation}{0}
\setcounter{section}{0}


%
%
%
%

\section{Introduction}

In this paper, we will study surfaces smoothly embedded in $4$-manifolds. For ease of notation, we will always work in the smooth category -- ``embedding'' or ``immersion'' should be taken to mean ``smooth embedding,'' or ``smooth immersion,'' respectively. Similarly,``isotopy''  or ``homotopy'' should be taken to mean ``smooth ambient isotopy,'' or ``smooth ambient homotopy,'' respectively.

We will prove an analogue of Gabai's recent 4-dimensional light bulb theorem \cite{gabai} in the setting of concordance.  We discuss the 4-dimensional light bulb theorem in length in Section \ref{4dsection}; for now it suffices to say that the 4-dimensional light bulb theorem regards when homotopic surfaces embedded in a $4$-manifold are actually isotopic (given some other hypotheses).

\begin{definition}[Concordance]
Let $M^m$ and $N^m$ be $m$-dimensional submanifolds of $X^n$. We say that $M^m$ and $N^m$ are {\emph{concordant}} if there exists an $(m+1)$-dimensional submanifold $H$ of $X^n\times I$ so that $H\cap(X^n\times 0)=M^m$, $H\cap(X^n\times 1)=N^m$, and $H\cong M^m\times I$. We call $H$ a {\emph{concordance}} from $M^m$ to $N^m$.
\end{definition}

The author of this paper previously wrote a note showing how the 4-dimensional light bulb theorem in $S^2\times S^2$ can be used to construct concordances in $S^2\times S^2$ \cite{otherpaper}. That is, how to construct a concordance from $R$ to $R'$ when $R, R'$ are genus-$g$ surfaces in the homology class $[S^2\times\pt]$ in $S^2\times S^2$. In this setting, the existence of such a concordance is already known from the following theorem of Sunukjian \cite{sunukjian}.

\begin{theorem}[{\cite[Thm. 6.1]{sunukjian}}]\label{sunukjianthm}
Let $X^4$ be a simply connected $4$-manifold. Then surfaces $S, S'$ in $X^4$ are concordant if and only if they have the same genus and $[S]=[S']$ in $H_2(X^4)$.
\end{theorem}

In this paper, we extend the construction of \cite{otherpaper} to to a more general $4$-dimensional manifold, as in the $4$-dimensional light bulb theorem. In particular, we do not assume the ambient manifold is simply connected, so our first main theorem does not follow from Theorem \ref{sunukjianthm} (unless $\pi_1(X^4)=0$).

\begin{theorem}\label{easytheorem}
Let $X^4$ be an orientable $4$-manifold so that $\pi_1(X^4)$ has no $2$-torsion. Let $R$, $R'$, and $G$ be embedded $2$-spheres in $X^4$ so that the following conditions hold:
\begin{itemize}
\item $G$ has trivial normal bundle,
\item $R\cap G=\pt$, where the intersection is transverse.
\item $R'$ is homotopic to $R$.
\end{itemize}

Then $R$ and $R'$ are concordant.
\end{theorem}

%
%
%


\begin{definition}
Let $R$ be a surface embedded in a $4$-manifold $X^4$. We say $G$ is a {\emph{transverse sphere to $R$}} if $G$ has trivial normal bundle and intersects $R$ in one point (transversely).
\end{definition}

Schwartz \cite{hannah} has shown that the $2$-torsion assumption of Theorem \ref{easytheorem} is necessary. In \cite{hannah}, Schwartz constructs a pair of homotopic spheres $R,R'$ with a common transverse sphere $G$ in $4$-manifold $X^4$ so that $R$ and $R'$ are not concordant. (The main purpose of this example is to show that a $2$-torsion hypothesis is necessary in the $4$-dimensional light bulb theorem to conclude that $R$ and $R'$ are isotopic, but this simultaneously obstructs the more general relation of concordance.) Schwartz has infinitely many distinct examples of such pairs (specifically, a finite number of examples in each of inifinitely many ambient $4$-manifolds). In these examples, the 4-dimensional light bulb theorem does not apply because $\pi_1(X^4)$ has $2$-torsion.

In fact, the counterexamples of \cite{hannah} also show that Theorem 6.2 of \cite{sunukjian} is not true. This theorem implies concordance of surfaces $S_0$ and $S_1$ in $4$-manifold $X^4$ given three conditions:
\begin{itemize}
\item $\pi_1(S_i)\to \pi_1(X^4)$ is trivial for each $i$ (e.g. $S_i$ is a sphere),
\item $[S_0]=[S_1]$ in $H_2(X^4,\Z[\pi_1])$ (i.e. the lifts of $S_0, S_1$ to the universal cover of $X^4$ are componentwise homologous)
\item There exists a third surface $S$ in $X^4$ so that $\pi_1(S)\to\pi_1(X)$ is trivial, $[S]=[S_0]$ in $H_2(X^4;\Z[\pi_1])$, and the meridian of $S$ is nullhomotopic in $X-S$ (e.g. $S=S_0$ if $S_0$ has a transverse sphere).
\end{itemize}

However, Schwartz's spheres $R$ and $R'$ (taking the place of $S_0$ and $S_1$) satisfy all three of these conditions but are not concordant. (The first condition is obviously satisfied and the third follows from $R$ and $R'$ having transverse spheres. The second condition appears during the construction of \cite{hannah}; the lifts of $R$ and $R'$ to the universal cover of $X^4$ are in fact isotopic.)

Otherwise, Theorem 6.2 of \cite{sunukjian} would imply Theorem \ref{easytheorem}. In the long term, we hope that some modification in the presence of 2-torsion might correct this theorem. Theorem \ref{easytheorem} (and Theorem \ref{maintheorem}, which has yet to be stated) cover the cases in which $S_0$ and $S_1$ are homotopic and $S_0$ has a transverse sphere.

Our second main theorem applies when $\pi_1(X^4)$ has $2$-torsion.

\begin{theorem}\label{maintheorem}
Let $X^4$ be an orientable $4$-manifold. Let $R$ and $R'$ be $2$-spheres embedded in $X^4$ so that $R$ has a transverse sphere $G$ and $R'$ is homotopic to $R$.

Then up to an obstruction related to how a homotopy from $R'$ to $R$ interacts with $2$-torsion elements in $\pi_1(X^4)$, $R'$ is concordant to $R$.
\end{theorem}

Theorem \ref{maintheorem} generalizes Theorem \ref{easytheorem}. We state Theorem \ref{maintheorem} precisely in Section \ref{thmstatement}, after giving several necessary definitions.

Finally, applying an argument of Sunukjian \cite[Thm. 7.5]{sunukjian}, we obtain the following corollary.

\begin{corollary}\label{scobordism}
Let $X^4$ be a $4$-manifold. Let $R$ and $R'$ be $2$-spheres smoothly embedded in $X^4$ satisfying the hypotheses of Theorem \ref{maintheorem}.

Assume that $\pi_1(X^4)$ is a good group (as in \cite{quinn}) and that a meridian of $R'$ is nullhomotopic in $X^4\setminus R'$.

There there is a homeomorphism from the pair $(X^4,R')$ to $(X^4,R)$.
\end{corollary}

This corollary is our only foray outside of the smooth category.

\begin{proof}

We repeat the argument almost verbatim. In Theorem \ref{maintheorem}, we construct a concordance $H$ from $R'$ to $R$. Lift $H\subset X\times I$ to the universal cover $\newtilde{X}\times I$ of $X\times I$ to find a cobordism $\newtilde{H}$ from $\newtilde{R'}$ to $\newtilde{R}$ (which are the lifts of $R'$ and $R$, respectively.) Each component of $\newtilde{R}$ has a transverse sphere in the lift $\newtilde{G}$ of $G$, so every meridian of a component of $\newtilde{H}$ bounds a disk in $\newtilde{H'}:=(\newtilde{X}^4\times I)\setminus\newtilde{H}$. Therefore, $\newtilde{H'}$ is the universal cover of $H':=(X^4\times I)\setminus H$.
The Mayer-Vietoris sequence says that $\newtilde{H'}$ is an $h$-cobordism (here using the fact that $\newtilde{X^4}\setminus\newtilde{R'}$, $\newtilde{X^4}\setminus\newtilde{R}$, and $\newtilde{H'}$ are simply connected), so $H'$ is also an $h$-cobordism. By additivity of Whitehead torsion (note $H$ and $H\cup H'=X^4\times I$ are products), $H'$ is actually an $s$-cobordism. Since $\pi_1(X^4\setminus R)\cong\pi_1(X^4)$ is good, $H'$ is topologically a product. This product structure yields the desired homeomorphism.
\end{proof}

\subsection*{Organization}
We organize the paper as follows.

{\bf{Section \ref{3dsection}:}} We discuss the $3$-dimensional light bulb theorem as lower-dimensional motivation. We give the proof of \cite{yildiz} of a concordance analogue of the $3$-dimensional light bulb theorem (also proved by \cite{davis}).

{\bf{Section \ref{4dsection}:}} We discuss the $4$-dimensional light bulb theorem and the statement of the main theorems.
\begin{itemize}
\item[]{\bf{Subsection \ref{simplesubsec}:}} We state the $4$-dimensional light bulb theorem and Theorem \ref{easytheorem} for $4$-manifolds with no $2$-torsion in $\pi_1$.
\item[]{\bf{Subsection \ref{htpysec}:}} We remind the reader of important facts about homotopy and regular homotopy of surfaces in $4$-manifolds.
\item[]{\bf{Subsection \ref{thmstatement}:}} We state the general $4$-dimensional light bulb theorem and Theorem \ref{maintheorem}. To do this, we give many definitions from \cite{gabai}.
\item[]{\bf{Subsection \ref{tubedsec}:}} We recall the definition of a tubed surface from  \cite{gabai}. 
\end{itemize}

{\bf{Section \ref{proofsec}:}} We give the proofs of Theorems \ref{easytheorem} and \ref{maintheorem}.

\subsection*{Acknowledgements}
The author thanks her graduate advisor, David Gabai, for helpful conversations. Thanks to Danny Ruberman for useful discussion on regular homotopy and understanding Theorem \ref{hirsch} (by Hirsch \cite{hirsch}) and Nathan Sunukjian for interesting comments, including suggesting Corollary \ref{scobordism}.
 
The author is a fellow in the National Science Foundation Graduate Research Fellowship program, under Grant No. DGE-1656466.

\section{A dimension down: the light bulb theorem in dimension three}\label{3dsection}

\begin{theorem}[3-dimensional light bulb theorem (folklore)]\label{3dlight bulb}
Let $K$ be a circle embedded in $S^1\times S^2$ so that $K$ intersects $\pt\times S^2$ geometrically once (and that intersection is transverse). Then $K$ is isotopic to $S^1\times\pt$.
\end{theorem}

\begin{proof}
Note that $K$ is regularly homotopic to $S^1\times \pt$ via a finite sequence of isotopies and crossing changes. The effect of each crossing change can be achieved via isotopy, by sweeping a strand of $K$ parallel to $\pt\times S^2$ (see Fig. \ref{fig:light bulbtrick}; this is called the ``light bulb trick'').  Thus, $K$ is in fact isotopic to $S^1\times\pt$.
\end{proof}

\begin{figure}
\includegraphics[width=.6\textwidth]{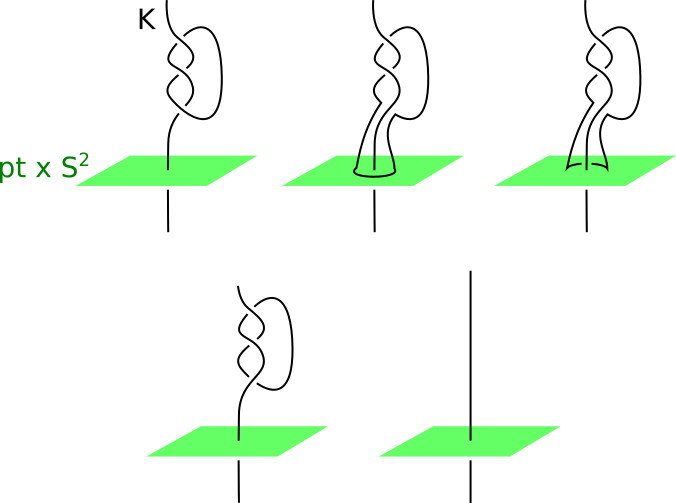}
\caption{Left to right, top to bottom: the $3$-dimensional light bulb trick. If $K$ is a knot in $S^1\times S^2$ intersecting $\pt\times S^2$ in a single point, then we can effect crossing changes in $K$ by sweeping a strand of $K$ over the $2$-sphere $\pt\times S^2$.}\label{fig:light bulbtrick}
\end{figure}

The statement of the 3-dimensional light bulb theorem requires that $K$ have a transverse sphere. In this dimension, we mean that there must be a $2$-sphere $G$ (in this setting, $\pt\times S^2$) so that $K\cap G=\pt$. When we only know the algebraic intersection of $K$ and $G$, then the theorem does not necessarily hold (for example, the knot $K$ in the leftmost frame of Fig. \ref{fig:3dconcordance} is not isotopic to $S^1\times\pt$, since $\pi_1((S^1\times~S^2)\setminus K)\not\cong \mathbb{Z}$). However, one can still construct a concordance from $K$ to $S^1\times\pt$.

\begin{theorem}[Concordance analogue of 3-dimensional light bulb theorem \cite{yildiz} \cite{davis}]
Let $K$ be a circle embedded in $S^1\times S^2$ so that $[K]=[S^1\times \pt]$ in $H_1(S^1\times S^2)$. Then $K$ is concordant to $S^1\times\pt$.
\end{theorem}

\begin{proof}
We illustrate this proof in Figure \ref{fig:3dconcordance}. This argument is due to Yildiz \cite{yildiz}.

\begin{figure}
\includegraphics[width=\textwidth]{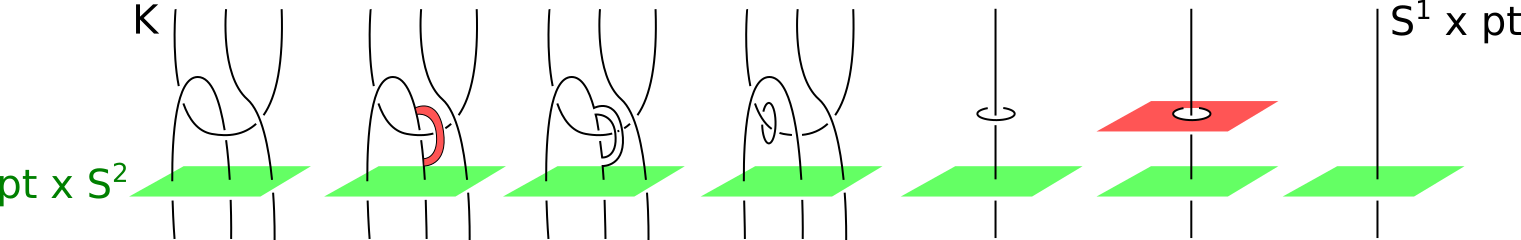}
\caption{Left to right: a movie of a concordance from $K$ to $S^1\times\pt$, where $K$ is a knot in $S^1\times S^2$ with $[K]=[S^1\times\pt]\in H_1(S^1\times D^2)$. The index-$1$ critical points of the concordance annulus appear as bands which effect crossing changes (as in the second frame). The index-$2$ critical points of the concordance annulus appear as disks capping off unlinked, unknotted components of one cross-section (as in the sixth frame). }\label{fig:3dconcordance}
\end{figure}

As in the 3-dimensional light bulb theorem, $K$ is regularly homotopic to $S^1\times\pt$. This regular homotopy is a finite sequence of isotopies and crossing changes. We build a concordance from $K$ (in this dimension, this means we build an annulus) in $(S^1\times S^2)\times I$. Because visualizing submanifolds via movies is an important concept in paper, we attempt to give this easy example in detail.

Say $K$ is regularly homotopic to $S^1\times\pt$ via homotopy $f:S^1\times I\to S^1\times S^2\times I$. Perturb $f$ so that all crossing changes happen at the same time, far from each other. Say there are $n$ crossing changes. Then take $f(S^1\times t)$ to be smoothly embedded for $t\neq 1/2+\epsilon$, and $f(S^1\times (1/2+\epsilon))$ to have $n$ self-intersections consisting of double-points (for some small $\epsilon>0$).

Now we build an annulus $A$ in $(S^1\times S^2)\times I$. Obtain $A'$ from $f(S^1\times [0,1/2])$ by attaching $2$-dimensional $1$-handles (bands) $b_1,\ldots, b_n$ to $f(S^1\times 1/2)$. Specifically, attach one band for each crossing change of $f$. Each band lives in a neighborhood of the corresponding crossing change in $f(S^1\times 1/2)$ and is embedded as in Figure \ref{fig:3dconcordance} (second image).

Let $K'=(f(S^1\times1/2)\setminus (\cup_i b_i))\cup\overline{((\cup_i \boundary b_i)\setminus f(S^1\times1/2))}$. (In words, $K'$ is obtained from $f(S^1\times1/2)$ by deleting the intersection with $\boundary b_i$ and adding in the rest of the boundary of $b_i$; this is normal band surgery.) View $K'$ as a subset of $S^1\times S^2$ by identifying $S^1\times S^2\times 1/2$ with $S^1\times S^2$. Note that $K'$ is a disjoint union of $n+1$ circles. One of these components $C$ is isotopic to $S^1\times\pt$, while other $n$ components $U_1,\ldots, U_n$ are meridians of $C$. Each $U_i$ bounds a disk $D_i$ which does not intersect $K'$ in its interior (with $D_i\cap D_j=\emptyset$ for $i\neq j$).

Now let $A'':=A'\cup (C'\times[1/2,3/4])$ in $S^1\times S^2$. Attach the $2$-dimensional $2$-handle (disk) $D_i$ to $(U_i\times 3/4)\subset\boundary A''$ for $i=1,\ldots, n$; call the result $A'''$. (See Fig. \ref{fig:3dconcordance}, sixth image.) Finally, let $A=A'''\cup_{t\in[0,1]}(g_t(S^1)\times(3+t)/4)$, where $g:S^1\times I\to S^1\times S^2\times I$ is an isotopy from $C$ to $S^1\times\pt$.

Thus, we have constructed a surface $A$ in $S^1\times S^2\times I$. See Figure \ref{fig:3dconcordance} for a clear schematic of the above construction. We remark that as described, $A$ is not smoothly embedded, but rather has corners ($A$ is in {\emph{horizontal-vertical position}}). We can standardly smooth these corners and take $A$ to be smoothly embedded. We won't remark on this distinction later in the paper.

By construction, $A\cap(S^1\times S^2\times 0)=K$ and $A\cap(S^1\times S^2\times 1)=S^1\times\pt$. Moreover, $A$ is obtained from $K\times[0,1/2]$ by attaching $n$ geometrically cancelling pairs of $1$- and $2$-handles, so $A$ is an annulus. Therefore, $A$ is a concordance from $K$ to $S^1\times\pt$.
\end{proof}

\section{The light bulb theorem in dimension four}\label{4dsection}
Now we move up a dimension, to consider the $4$-dimensional light bulb theorem.

\subsection{When the ambient $4$-manifold has no $2$-torsion in its fundamental group}\label{simplesubsec}

\begin{theorem}[4-dimensional light bulb theorem, {\cite[Thm. 1.2]{gabai}}]\label{easy4dlight}
Let $X^4$ be an orientable $4$-manifold so that $\pi_1(X^4)$ has no $2$-torsion. Let $R$ and $R'$ be $2$-spheres embedded in $X^4$ so that $R$ and $R'$ have a mutual transverse sphere $G$ and $R$ is homotopic to $R'$. Then $R$ and $R'$ are isotopic.
\end{theorem}

In fact, the above theorem also states that the isotopy from $R$ to $R'$ can be taken to fix a neighborhood of $G$ if $R$ and $R'$ coincide near $G$.

To compare the $4$-dimensional light bulb theorem with the 3-dimensional light bulb theorem, one should notice that $G$ takes the role of $\pt\times S^2$, $R$ takes the role of $S^1\times \pt$ and $R'$ take the role of $K$. The proof of the $4$-dimensional light bulb theorem is considerably more involved than the proof of the $3$-dimensional light bulb theorem.

In Section \ref{3dsection}, we discussed a concordance analogue of the 3-dimensional light bulb theorem. In this paper, we give a concordance analogue of the 4-dimensional light bulb theorem.

\begin{easytheorem}
Let $X^4$ be an orientable $4$-manifold so that $\pi_1(X^4)$ has no $2$-torsion. Let $R$ and $R'$ be $2$-spheres embedded in $X^4$ so that $R$ has a transverse sphere $G$ and $R$ is homotopic to $R'$. Then $R$ and $R'$ are concordant.
\end{easytheorem}

Gabai \cite{gabai} produces a more general version of the 4-dimensional light bulb theorem that may apply even when $\pi_1(X^4)$ has $2$-torsion. To state this theorem, we must first understand regular homotopy of surfaces.

\subsection{Regular homotopy of surfaces in $4$-manifolds}\label{htpysec}

In the 4-dimensional light bulb theorem, one need not distinguish between homotopy and regular homotopy due to the following celebrated theorem of Smale \cite{smale}.

\begin{theorem}[{\cite[Thm. D]{smale}}]\label{smale}
Two smooth embedded $2$-spheres in an orientable $4$-manifold $X^4$ are homotopic if and only if they are regularly homotopic.
\end{theorem}

The above theorem is actually stated more generally for immersions of $2$-spheres in $X^4$ (with an extra restriction), but we need only concern ourselves with the statement for homotopy between embedded surfaces in this paper. A similar result of Hirsch \cite{hirsch} holds for homotopic positive-genus surfaces.

\begin{theorem}[{\cite[Thm. 8.3]{hirsch}}]\label{hirsch}
Two smooth embedded orientable genus-$g$ surfaces  in an orientable $4$-manifold $X^4$ are homotopic if and only if they are regularly homotopic.
\end{theorem}

The cited theorem is actually stated for immersed $k$-spheres in $2k$-manifolds, but the proof carries through in this setting as well -- we first isotope the surfaces to agree outside of a disk, and then apply the arguments of \cite[Lemma 8.1 and Theorem 8.3]{hirsch} to regularly homotope the remaining disk.

See \cite{quinn} for more exposition on the following definitions and well-known proposition about regular homotopy.

%

\begin{definition}\label{fingerdef}
Let $S$ be a surface smoothly immersed in $4$-manifold $X^4$. Let $\gamma$ be an arc in $X^4$ with endpoints on $S$ so that $\int{\gamma}\cap S=\emptyset$. Take $\boundary\gamma$ to be far from self-intersections of $S$. A {\emph{finger move along $\gamma$}} is the regular homotopy of $S$ which homotopes one disk component of $\nu(\gamma)\cap S$ along $\gamma$ to create a new pair of oppositely-signed transverse intersections of $S$. See top of Figure \ref{fig:reghomotopy}.

We will usually name this finger move $f_i$ for some index $i$. Then we will write $\gamma_i$ to indicate the arc $\gamma$ along which the finger move takes place.
\end{definition}

\begin{figure}
\includegraphics[width=.65\textwidth]{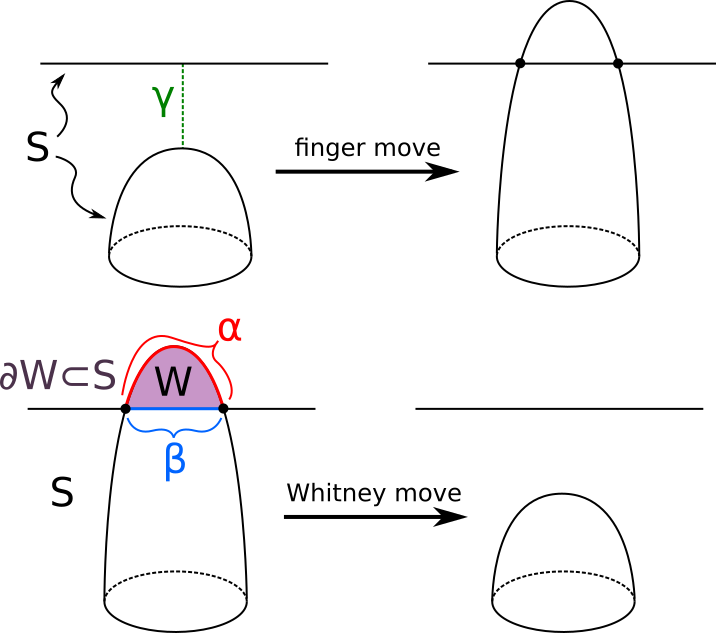}
\caption{Top row: a finger move along arc $\gamma$, as in Def. \ref{fingerdef}. Bottom row: a Whitney move along disk $W$, as in Def. \ref{whitneydef}. We may take $N_\boundary$ to point out of the page.}\label{fig:reghomotopy}
\end{figure}

\begin{definition}\label{whitneydef}
Let $S$ be a surface smoothly immersed in $4$-manifold $X^4$. Let $x$ and $y$ be two distinct self-intersections of $S$, of opposite sign. Let $\alpha$ and $\beta$ be arcs embedded in $S$ from $x$ to $y$ so that $\int{\alpha}$ and $\int{\beta}$ do not meet self-intersections of $S$, and near $x$ and $y$ the arcs $\alpha,\beta$ live in different sheets of $S$. Assume there exists a disk $W$ in $X^4$ with $\boundary W=\alpha\cup\beta$ and $\int{W}\cap S=\emptyset$. Then $\boundary W$ comes with a preferred $1$-dimensional subbundle $N_\boundary$ of its normal bundle ($N_\boundary$ is normal to $W$; along $\alpha$, $N_\boundary$ is parallel to the sheet of $S$ containing $\alpha$; along $\beta$, $N_\boundary$ is normal to the sheet of $S$ containing $\beta$). If a non-vanishing section of $N_\boundary$ extends to the whole normal bundle of $W$ (this is a $\mathbb{Z}/2\mathbb{Z}$ obstruction), then we say that $W$ is a {\emph{Whitney disk}}.

A {\emph{Whitney move along $W$}} is the regular homotopy of $S$ which homotopes the sheet of $S$ containing $\beta$ along $W$ to remove the self-intersections $x$ and $y$. See bottom of Figure \ref{fig:reghomotopy}. (The extension of $N_\boundary$ to the normal bundle of $W$ is a technical requirement to make this move possible.)

We will usually name this Whitney move $w_i$ for some index $i$. Then we will write $W_i$ to indicate the Whitney disk along which this Whitney move takes place.
\end{definition}

\begin{remark}\label{inverseremark}
The finger move and Whitney move are inverse operations. That is, let $f$ be a finger move, so $f$ is a regular homotopy from surface $S$ to surface $S'$, where $S$ has $n$ self-intersections and $S'$ has $n+2$ self intersections. Let $\overline{f}$ be the inverse of $f$, so $\overline{f}$ is a regular homotopy from $S'$ to $S$ which cancels the two self-intersections introduced by $f$. The homotopy $\overline{f}$ is a Whitney move, as implicitly illustrated in Figure \ref{fig:reghomotopy}.

Similarly, when $w$ is a Whitney move from $S$ to $S'$, then the inverse homotopy $\overline{w}$ is a finger move from $S'$ to $S$.
\end{remark}
\begin{notation}
Usually, we will write a Whitney move from $S$ to $S'$ as $w_i$ for some index $i$. Then we will refer to the path along which the finger move $\overline{w_i}$ takes place as $\eta_i$. Recall that $\eta_i$ is an arc in $X^4$ with endpoints on $S'$ with $\int{\eta}_i\cap S'=\emptyset$.
\end{notation}

\begin{proposition}
Let $S$ and $S'$ be smoothly embedded surfaces in the smooth $4$-manifold $X^4$. Suppose $S$ is regularly homotopic to $S'$.  Then up to isotopy, the regular homotopy can be obtained by a finite sequence of finger moves followed by a finite sequence of Whitney moves, with intermediate isotopy at each step.
\end{proposition}

\begin{notation}
Let $h$ be a regular homotopy of a surface $R'$ which consists of finger moves $f_1,\ldots, f_n$ followed by Whitney moves $w_1,\ldots, w_n$ with intermediate isotopy. By slight abuse of notation, we will always refer to $f_i$ as a finger move of $R'$ for any $i$.

We will always denote the surface obtained from $R'$ by doing only the finger moves $f_1,\ldots, f_n$ by $S$.
For $i=2,\ldots, n$, we let $S_i$ denote the surface obtained from $S$ by performing Whitney moves $w_1,\ldots, w_{i-1}$. Then $w_i$ is a regular homotopy of $S_i$, so $\boundary W_i\subset S_i$. Note by dimensionality that $W_i$ does not intersect $\eta_j$ for any $j<i$.
\end{notation}

\subsection{When the ambient $4$-manifold has nontrivial $2$-torsion in its fundamental group.}\label{thmstatement}
We now move onto specific definitions required to parse the statement of the generalized 4-dimensional light bulb theorem. From now on, let $A$ abstractly be the $2$-sphere. Given a sphere $Y$ embedded or immersed in $X^4$, let $\pi_Y:A\to X^4$ be the actual embedding or immersion.

\begin{remark}\label{remark2.2}
Let $S$ be a $2$-sphere embedded in $X^4$. Fix points $x$ and $y$ in $S$. Let $\gamma$ be an arc in $X^4$ from $x$ to $y$ with $\int{\gamma}\cap S=\emptyset$. Fix a point $z$ in $S$. Let $\gamma_{zx},\gamma_{yz}$ be arcs in $S$ from $z$ to $x$ and $y$ to $z$, respectively. Then $\gamma$ uniquely determines the element of $\pi_1(X^4,z)$ represented by $\gamma_{zx}\gamma\gamma_{yz}$. We write $[\gamma]$ to indicate this element of $\pi_1(X^4,z)$.
\end{remark}

\begin{definition}\label{labeldef}
Let $R$ and $R'$ be regularly homotopic $2$-spheres in $X^4$. Say that some regular homotopy $h$ from $R'$ to $R$ consists of the finger moves $f_1,\ldots, f_n$ followed by the Whitney moves $w_1,\ldots, w_n$ (with intermediate isotopies). Let $S$ be the surface obtained from $R'$ after performing the finger moves $f_1,\ldots, f_n$, so $S$ is a $2$-sphere immersed in $X^4$ with $2n$ points of self-intersection.

Let $(x_1,y_1),\ldots,(x_{2n},y_{2n})$ be pairs of distinct points in $A$ mapping to distinct self-intersections of $S$, so $\pi_S(x_i)=\pi_S(y_i)$. We take $\pi_{S_j}(x_i)=\pi_{S_j}(y_i)$ if $\pi_S(x_i)=\pi_S(y_i)$ is not cancelled by one of the Whitney moves $w_1,\ldots, w_{j-1}$. For each finger move $f_i$, there exist two distinct $i_1,i_2$ so that $\pi_S(x_{i_1},x_{i_2},y_{i_1},y_{i_2})$ lie in the support of $f_i$ (i.e. are the two self-intersections introduced by $f_i$). Choose the labelings of each pair $(x_j, y_j)$ so that $x_{i_1}$ and $x_{i_2}$ lie in the same sheet of $S$ in this support, while $y_{i_1}$ and $y_{i_2}$ lie on the other.

We now refer to $L=(L_x,L_y)$ as a {\emph{labeling}} of $h$, where $L_x=\{x_1,\ldots, x_n\}, L_y=\{y_1,\ldots, y_n\}$. There are $2^n$ distinct labelings of $h$.
\end{definition}

\begin{definition}[{\cite[Def. 5.15]{gabai}}]\label{crosseddef}
Let $R$ and $R'$ be regularly homotopic $2$-spheres in $X^4$. Let $h$ be a regular homotopy from $R'$ to $R$ consisting of the finger moves $f_1,\ldots, f_n$ followed by the Whitney moves $w_1,\ldots, w_n$ (with intermediate isotopies). Let $L=(L_x,L_y)$ be a labeling of $h$, where $L_x=\{x_1,\ldots,x_n\}, L_y=\{y_1,\ldots, y_n\}\subset A$.

Let $W_i$ be the Whitney disk associated to Whitney move $w_i$, as in Definition \ref{whitneydef}. Then $\pi_{S_i}^{-1}(\boundary W_i)$ consists of two arcs $\boundary_i^1$ and $\boundary_i^2$, whose four boundary points collectively consist of two points in $L_x$ and two points in $L_y$.

If one of $\boundary_i^1,\boundary_i^2$ connects two points in $L_x$ while the other connects two points in $L_y$, then we say that $w_i$ is {\emph{uncrossed}} with respect to $L$. If each of $\boundary_i^1,\boundary_i^2$ meet both $L_x$ and $L_y$, then we say that $w_i$ is {\emph{crossed}} with respect to $L$.
\end{definition}

Now we are able to state the general 4-dimensional light bulb theorem.

\begin{theorem}[4-dimensional light bulb theorem, {\cite[Thm. 1.3]{gabai}}]\label{general4d}

Let $X^4$ be an orientable $4$-manifold. Let $R$ and $R'$ be $2$-spheres embedded in $X^4$ so that $R$ and $R'$ have a mutual transverse sphere $G$ and $R$ is homotopic to $R'$.

Let $h$ be a regular homotopy (via Thm. \ref{smale}) from $R'$ to $R$ which consists of a sequence of finger moves $f_1,\ldots, f_n$ followed by Whitney moves $w_1,\ldots, w_n$ (with intermediate isotopies). Choose a labeling $L$ of $h$. 

Let $\eta_i$ be the path along which the finger move $\overline{w_i}$ takes place, as in Remark \ref{inverseremark}. By Remark \ref{remark2.2}, each $\eta_i$ represents an element $[\eta_i]$ of $\pi_1(X^4)$ with basepoint on $R$. Let $\mathcal{H}$ be the multiset $\{[\eta_i]\mid w_i$ is crossed with respect to $L\}$. 

If each $2$-torsion element of $\pi_1(X^4)$ appears an even number of times in $\mathcal{H}$, then $R'$ is isotopic to $R$.
\end{theorem}

From now on, when we refer to the ``4-dimensional light bulb theorem,'' we mean the statement of Theorem \ref{general4d}. When $\pi_1(X^4)$ has no $2$-torsion, this restricts to the statement of Theorem \ref{easy4dlight}.

The statement we give here of Theorem \ref{general4d} does not exactly match the statement given in \cite{gabai}; we imagine a reader referring from this paper to \cite{gabai} for the first time might be confused by the difference. Gabai states Theorem \ref{general4d} in terms of putting $R'$ into a normal form with respect to $R$ and then later gives the exact obstruction to this normally positioned surface being isotopic to $R$. Here, we have pulled back the whole result into one statement to make later discussion easier. 

We extend Theorem \ref{easytheorem} accordingly.

\begin{maintheorem}

Let $X^4$ be an orientable $4$-manifold. Let $R$ and $R'$ be $2$-spheres embedded in $X^4$ so that $R$ has a transverse sphere $G$ and $R$ is homotopic to $R'$.

Let $h$ be a regular homotopy (via Thm. \ref{smale}) from $R'$ to $R$ which consists of a sequence of finger moves $f_1,\ldots, f_n$ followed by Whitney moves $w_1,\ldots, w_n$ (with intermediate isotopies). Choose a labeling $L$ of $h$.

Let $\eta_i$ be the path along which the finger move $\overline{w_i}$ takes place, as in Remark \ref{inverseremark}. By Remark \ref{remark2.2}, each $\eta_i$ represents an element of $\pi_1(X^4)$ with basepoint on $R$. Let $\mathcal{H}$ be the multiset $\{[\eta_i]\mid w_i$ is crossed with respect to $L\}$. 

If each $2$-torsion element of $\pi_1(X^4)$ appears an even number of times in $\mathcal{H}$, then $R'$ is concordant to $R$.

\end{maintheorem}

To prove Theorem \ref{maintheorem}, we must understand some details of the proof of the $4$-dimensional light bulb theorem.

\subsection{Tubed surfaces}\label{tubedsec}

In \cite{gabai}, Gabai defines the class of {\emph{tubed surfaces}}, which are defined by attaching tubes to embedded surfaces in a prescribed way.

\begin{definition}[{\cite[Def. 5.4]{gabai}}]\label{frameddef}
A {\emph{framed embedded path}} is a smooth embedded path $\tau:I\to X^4$ with a framing $(\nu_1(t),\nu_2(t),\nu_3(t))$ of its normal bundle. Let $C(t)$ be a circle bounding a disk centered at $\tau(t)$ in the plane spanned by $\nu_1(t),\nu_2(t)$. Take each $C(t)$ to have small radius. We call the annulus $\cup_{t\in [0,1]} C(t)$ the {\emph{cylinder from $C(0)$ to $C(1)$}}.
\end{definition}

\begin{remark}\label{fingerframed}
In the definition of a finger move $f$ along $\gamma$ of an immersed surface $S$, the path $\gamma$ is actually a framed embedded path. The framing on $\gamma$ is chosen so that $C(0)\subset S$ and $C(1)$ intersects $S$ in two points. The result $S'$ of the finger move can be obtained from $S$ by deleting the disk in $S$ bounded by $C(0)$ which contains $\gamma(0)$, then attaching the cylinder from $C(0)$ to $C(1)$ and a small disk $D$ bounded by $C(1)$ chosen so that $\int{D}\cap S=\emptyset$ and the resulting surface has transverse self-intersections.
\end{remark}

\begin{definition}[{\cite[Def. 5.5]{gabai}}]\label{tubeddef}
Let $S$ be an immersed surface in $X^4$. Fix a transverse sphere $G$ for $S$. Say $S$ has $n$ points of self-intersection, so there are $2n$ distinct points $x_1,\ldots,x_n,y_1,\ldots,y_n\in A$ with $\pi_S(x_i)=\pi_S(y_i)$ for each $i$. Let $z_0=\pi_S^{-1}(z)$. A {\emph{tubed surface}} $S_T$ on $S$ consists of the following data:
\begin{enumerate}[i)]
\item The immersion $\pi_S:A\to X^4$.
\item For each $i=1,\ldots, n$, an immersed path $\sigma_i\subset A$ from $x_i$ to $z_0$.
\item Immersed paths $\alpha_1,\ldots,\alpha_r$ in $A$ with both endpoints at $z_0$ and for each $i=1,\ldots,r$, pairs of points $(p_i,q_i)$ in $A$ with $p_i\in\int{\alpha}_i$ and a framed embedded path $\tau_i\subset X^4$ from $\pi_S(p_i)$ to $\pi_S(q_i)$ with $\int{\tau}_i\cap(G\cup S)=\emptyset$.
\item Pairs of immersed paths $(\beta_1,\gamma_1),\ldots,(\beta_s,\gamma_s)$ in $A$ where $\beta_i$ goes from $z_0$ to $b_i$ and $\gamma_i$ goes from $g_i$ to $z_0$ (for some $b_i,g_i\in A$) and framed embedded paths $\eta_i$ from $\pi_S(b_i)$ to $\pi_S(g_i)$ with $\int{\eta}_i\cap(G\cup S)=\emptyset$.
\end{enumerate}

The union of all arcs  $\sigma_i,\alpha_j,\beta_k,\gamma_l,\tau_p,\eta_q$  is called the {\emph{tube guide locus}} of $S_T$.

We require that the $\sigma_i,\alpha_j,\beta_k,\gamma_l$ curves be self-transverse and transverse to each other, and that their interiors not meet any points of the form $x_i,q_j,b_k,g_l$, and also be disjoint from $z_0$ and the $p_i$ points except as specified. At crossings of these curves, one sheet should be labeled as above or below the other sheet (as in a crossing in a standard knot diagram). The points of the form $x_i,y_j,p_k,q_l,b_m,g_n$ are all distinct (and distinct from $z_0$).

The curves $\tau_i$ and $\eta_j$ are pairwise disjoint. We require they be normal to $S$ near their boundaries. Recall that for each framed arc $\tau_i$ or $\eta_i$, we defined circles $C(0)$ and $C(1)$ near their boundaries in Definition \ref{frameddef}. We restrict the allowed framings on $\tau_i$ and $\eta_j$, but do not state this condition until Construction \ref{realization}.

We say that $S$ is the underlying surface of $S_T$.
\end{definition}


From a tubed surface $S_T$ on $S$ we can construct an embedded surface.

\begin{construction}[{\cite[Construction 5.7]{gabai}}]\label{realization}
Let $S_T$ be a tubed surface on $S$. From $S_T$, we construct an embedded surface $S_R$ called the {\emph{realization}} of $S_T$ as follows (see Fig. \ref{fig:tubedsurface} for an illustration that is likely more helpful than the ensuing wall of text):

\begin{enumerate}[i)]
\item For each $i$, remove from $S$ a disk $\pi_S(\nu(y_i))$. Attach to this new boundary component a disk $D(\sigma_i)$ consisting of a tube that follows $\pi_S(\sigma_i)$ and connects to a copy of $G\setminus\int{\nu}(z)$, pushed slightly off $G$. 
\item For each $\alpha_i$ arc, let $P(\alpha_i)$ be a $2$-sphere obtained by attaching a copy of $G\setminus\int {\nu}(z)$ to each end of a tube following $\pi_S(\alpha_i)$, and pushing the copies of $G$ slightly off $G$ (and each other). The restriction on the framing of $\tau_i$ mentioned in Definition \ref{tubeddef} is that if $C(0)$ and $C(1)$ are the circles near $\tau_i(0)$ and $\tau_i(1)$ as in Definition \ref{frameddef}, then we require $C(0)$ to lie in $P(\alpha_i)$ and $C(1)$ to lie in $S$. Delete open disks in $P(\alpha_i)$ and $S$ bounded by $C(0)$ and $C(1)$ respectively and glue the resulting punctured surfaces together via the cylinder from $C(0)$ to $C(1)$ (as in Def. \ref{frameddef}). This yields an embedded surface $\newtilde{S}$. We call the cylinders around $\tau_i$ a {\emph{single tube}}.
\item Now for each $\eta_i$ arc, construct disks $D(\beta_i)$ and $D(\gamma_i)$ consisting of copies of $G\setminus(\int{\nu}(z))$ (pushed slightly off $G$ and each other) with collars parallel to $\pi_S(\beta_i)$ and $\pi_S(\gamma_i)$ respectively, so the boundary of $D(\beta_i)$ lies in a disk normal to $S$ at $\pi_S(b_i)$ and the boundary of $D(\gamma_i)$ lies in a disk normal to $S$ at $\pi_S(g_i)$. Fix $4$-balls $N_{b_i}$ and $N_{g_i}$ about $\pi_S(b_i)$ and $\pi_S(g_i)$ so that $\boundary N_{b_i}\cap(\newtilde{S}\cap D(\beta_i))$ is a Hopf link in the $3$-sphere $\boundary N_{b_i}$, and similarly $\boundary N_{g_i}\cap(\newtilde{S}\cap D(\gamma_i))$ is a Hopf link in the $3$-sphere $\boundary N_{b_i}$. 

 The restriction on the framing of $\eta_i$ mentioned in Definition \ref{tubeddef} is that if $C(0)$ and $C(1)$ are the circles near $\tau_i(0)$ and $\tau_i(1)$ as in Definition \ref{frameddef}, then we require $C(0)$ to lie in $\boundary N_{b_i}\cap\newtilde{S}$ and $C(1)$ to be $\boundary D(\eta_i)$. Let $x(t)\in C(0)$ be the point in direction $\nu_1(t)$ in the framing of $\eta_i$. (see Def. \ref{frameddef})
 
 Connect the specified Hopf links by two tubes parallel to $\eta_i$. One tube is the cylinder from $C(0)$ to $C(1)$ and connects $\boundary N_{b_i}\cap\newtilde{S}$ to $\boundary  D(\eta_i)$. The other tube is centered around $\cup_t x(t)$ and connects $\boundary  D(\beta_i)$ to $\boundary N_{g_i}\cap\newtilde{S}$. We call these two tubes together a  {\emph{double tube}}. The resulting embedded surface is $S_R$.
\end{enumerate}

At each stage, whenever two tube segments correspond to arcs of the tube guide locus which intersect in $A$, take the tube corresponding to the ``under'' segment to have smaller radius and thus lie closer to $S$, to avoid self-intersections of $S_R$. This is a slight abuse of notation, as one arc of the tube guide locus in $A$ may cross itself -- but simply take the piece of the tube corresponding to the ``under'' segment to be narrow.

We illustrate this construction in Figure \ref{fig:tubedsurface}.
\end{construction}

\begin{figure}
\includegraphics[width=.85\textwidth]{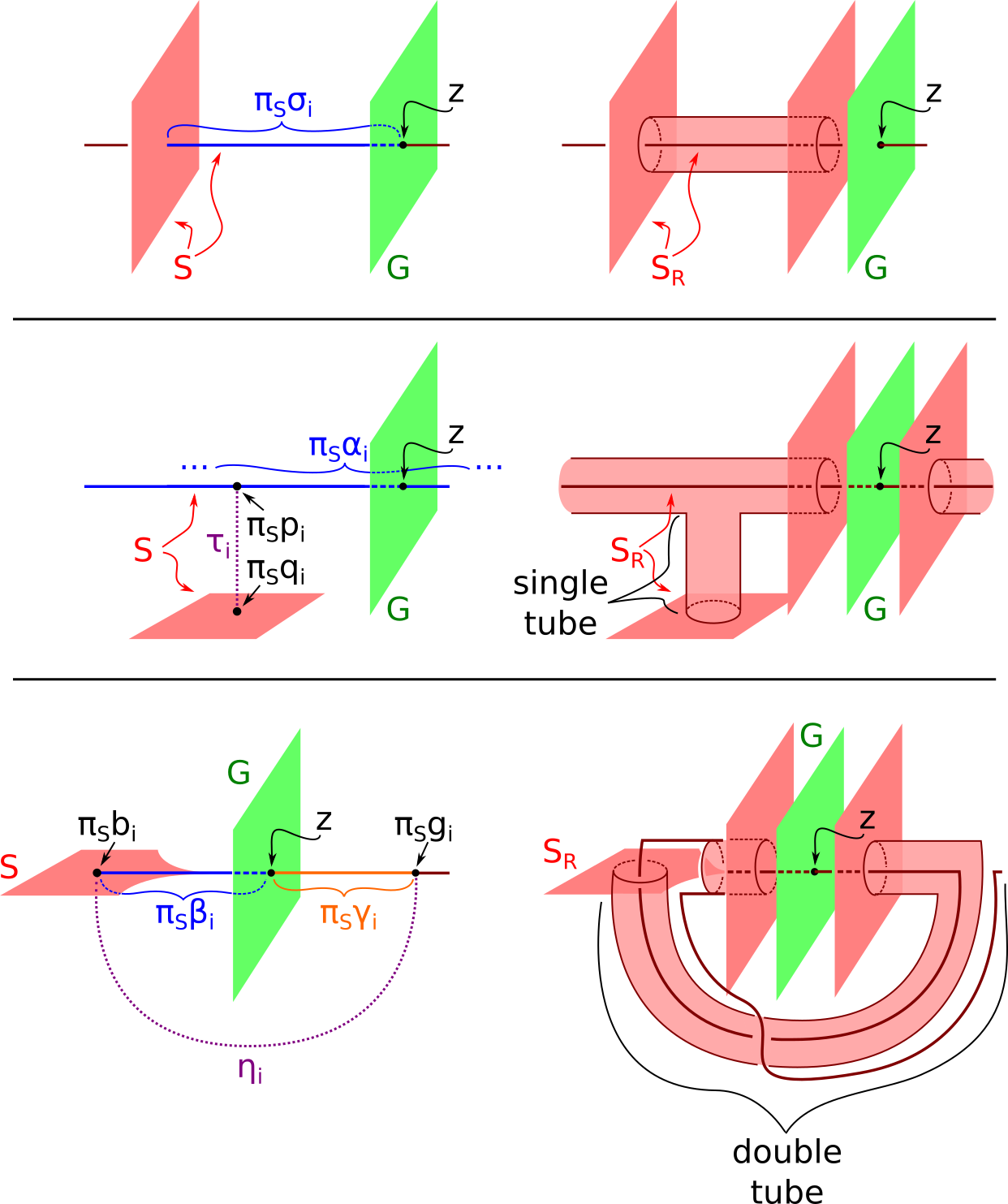}
\caption{Constructing the realization $S_R$ of a tubed surface on $S$. On the left, we draw a schematic of a $\sigma_i, \tau_i,$ or $\eta_i$ curve in the tubed surface tube guide locus. On the right, we draw the corresponding piece of $S_R$, as described in Construction \ref{realization}.}\label{fig:tubedsurface}
\end{figure}

A major part of the proof of the $4$-dimensional light bulb theorem is the following proposition.

\begin{proposition}[\cite{gabai}]\label{tubeprop}
Let $S_T$ be a tubed surface on $S$, where $S$ is a $2$-sphere embedded in $X^4$. Suppose that for each element $[\gamma]$ of $2$-torsion in $\pi_1(X^4)$, $[\gamma]$ appears an even number of times in the list $[\eta_1],\ldots,[\eta_k]$ where $\eta_1,\ldots,\eta_k$ are as in Definition \ref{tubeddef} (in words, $\eta_1,\ldots,\eta_k$ are the arcs yielding double tubes of $S_R$; recall by Remark \ref{remark2.2} that $[\eta_i]$ is an element of $\pi_1(X^4)$ with basepoint in $S$). Then $S_R$ is isotopic to $S$.
\end{proposition}

\section{Proof of Theorems \ref{easytheorem} and \ref{maintheorem}}\label{proofsec}

A very basic outline for the proof of Theorems \ref{easytheorem} and \ref{maintheorem} is as follows:
\begin{enumerate}[1.]
\item In $X^4\times I$, build a cobordism from $R'$ to a positive-genus surface $S_+$ by attaching $3$-dimensional $1$-handles to $R'\times I$.
\item Attach geometrically cancelling $3$-dimensional $2$-handles to the above cobordism to find a concordance from $R'$ to $R''$, where $R''$ is a sphere homotopic to $R'$ (and $R$) and $R''\cap G=\pt$. 
\item Argue that $R''$ is the realization of a tubed surface on $R$.
\item Apply Proposition \ref{tubeprop} to conclude that $R''$ is isotopic to $R$ given the hypotheses of Theorem \ref{maintheorem}.
\end{enumerate}


\subsection{Construction of a concordance from $R'$}\label{easyproof}

Let $z:=R\cap G$. Recall that $h$ is a regular homotopy from $R'$ to $R$ consisting of finger moves $f_1,\ldots, f_n$ followed by Whitney moves $w_1,\ldots, w_n$ (with intermediate isotopies), and with labeling $L$.

Let $S$ be the surface obtained from $R'$ by performing only the finger moves $f_1,\ldots, f_n$, so $S$ is an immersed $2$-sphere in $X^4$ with $2n$ points of self-intersection.

Recall that $\gamma_i$ is the path along which the finger move $f_i$ takes place, where $\gamma_i$ is an arc in $X^4$ with $\boundary\gamma_i\subset R'$ and $\int{\gamma_i}\cap R'=\emptyset$ as in Definition \ref{fingerdef}. Recall also that $W_i$ is the Whitney disk associated to $w_i$, so that $W_i$ is a disk in $X^4$ with $\boundary W_i\subset S_i$ and $\int{W_i}\cap S_i=\emptyset$ as in Definition \ref{whitneydef}. Again, $S_i$ denotes the surface obtained from $S$ after performing Whitney moves $w_1,\ldots, w_{i-1}$.




Let $L=(L_x,L_y)$ be a labeling of $h$ as in Definition \ref{labeldef}, so $L_x=\{x_1,\ldots, x_n\}, L_y=\{y_1,\ldots, y_n\}$ are disjoint subsets of $A$ with $\pi_S(x_i)=\pi_S(y_i)$. The map $\pi_S^{-1}$ takes $S\cap($support of $f_i$)) to two disjoint disks; in the definition of a labeling we require that two points in $L_x$ be contained in one of these disks and two points in $L_y$ be in the other.



Let $S_+$ be the genus-$n$ {\emph{embedded}} surface in $X^4$ obtained from $S$ by attaching a tube $T_i$ between the two self-intersections of $S$ created by $f_i$, as in Figure \ref{fig:splus} (left two images). Specifically, if $\pi_S(x_j)$ and $\pi_S(x_k)$ are in the support of $\gamma_i$ (for $j\neq k$), fix an arc $\sigma_i$ in $A$ from $x_j$ to $x_k$. Take $\sigma_1,\ldots,\sigma_n$ to be disjoint. Then for $i=1,\ldots, n$, delete $\pi_S(\int{\nu}(y_i))$ from $S$. Attach $n$ tubes $T_1,\ldots, T_n$ to this bounded surface, with $T_i$ parallel to $\pi_S(\sigma_i)$.




\begin{figure}
\includegraphics[width=\textwidth]{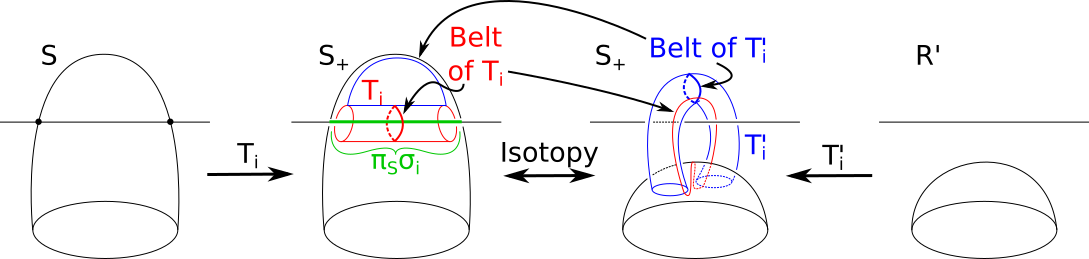}
\caption{{\bf{Left:}} the immersed surface $S$ is obtained from $R'$ by doing finger moves $f_1,\ldots, f_n$. {\bf{Second:}} We obtain $S_+$ from $S$ by surgery along tubes $T_1,\ldots, T_n$, where $T_i$ lies in the support of finger move $f_i$. {\bf{Third:}} We obtain $S_+$ from $R'$ by surgery along tubes $T'_i$, where $T'_i$ lies in the support of finger move $f_i$. {\bf{Right:}} The embedded surface $R'$.}\label{fig:splus}
\end{figure}

\begin{remark}
Although we have described $S_+$ as being obtained from $S$ by attaching tubes $T_i$ ($i=1,\ldots, n$), we can alternatively obtain $S_+$ from $R'$ by attaching tubes $T'_i$ ($i=1,\ldots, n$). See Figure \ref{fig:splus} (right two images). The tube $T'_i$ lies in the support of $\gamma_i$. For each $i=1,\ldots, n$, let $\newtilde{\gamma}_i:[0,1+\epsilon]$ be an extension of  $\gamma_i$, so that $\newtilde{\gamma}_i\vert_{[0,1]}=\gamma_i$, $\newtilde{\gamma}_i\vert_{(1,1+\epsilon]}\cap R'=\emptyset$, $\newtilde{\gamma}_i\cap\newtilde{\gamma_j}=\emptyset$ for $i\neq j$, and $\newtilde{\gamma}_i$ intersects $R'$ transversely at $\newtilde{\gamma}_i(1)$.

Let $D_i:=\newtilde{\gamma}_i\times I$ be contained in a small neighborhood of $\newtilde{\gamma}_i$, where the product direction is chosen so that $\newtilde{\gamma}_i(0)\times I\subset R'$ and $(\newtilde{\gamma}_i(0,1+\epsilon]\times I)\cap R'=\gamma_i(1)$. Let $\gamma'_i:=\boundary D_i\setminus(\newtilde{\gamma}_i(0)\times\int{I})$ and frame $\gamma'_i$ so that $C(0)$ and $C(1)$ are both contained in $R'$. Then let $T'_i$ be the cylinder from $C(0)$ to $C(1)$. We obtain $S_+$ from $R'$ by deleting the interiors of small disks bounded by $C(0)$ and $C(1)$ and then attaching $T'_i$, for $i=1,\ldots, n$.
\end{remark}

For each $\gamma_i$, let $H_i$ be a narrow solid tube $\gamma'_i\times D^2$, where the product direction is taken so that $\boundary H_i=(2$ disks in $R')\cup T'_i$. 
Let $M^3_1\subset X^4\times I$ be a cobordism from $R'$ to $S_+$ given by \[M^3_1=R'\times[0,1/2]\cup \sqcup_{i=1}^nH_i\times 1/2 \cup S_+\times [1/2,1].\]

%
%
%
%

We fix the above handle description of $M^3_1$, so that ``the $1$-handles of $M^3_1$'' will always refer to $H_1,\ldots, H_n$.


\begin{remark}\label{cancelremark}
Recall that $S_+$ is obtained from $S$ by attaching the tubes $T_1,\ldots, T_n$, where $T_i=\pi_S(\sigma_i)\times S^1$ for an arc $\sigma_i$ in $A$ from $x_{j}$ to $x_{k}$ (for some $j\neq k$). 
We refer to $\pi_S(\sigma_i(1/2))\times S^1$ as the {\emph{belt sphere}} of $T_i$. This belt intersects the belt sphere of the $1$-handle $H_i$ in exactly one point, and does not intersect the belt sphere of $H_l$ ($l\neq i$) at all. Then attaching $n$ $3$-dimensional $2$-handles to $M^3_1$ along annular neighborhoods of the belt spheres of $T_1,\ldots, T_n$ would geometrically cancel the $1$-handles of $M^3_1$.
\end{remark}


Recall that $\overline{w_i}$ is a finger move of $R$, and $\eta_i$ is the path along which this finger move takes place. Then $\eta_i$ is an arc in $X^4$ with $\boundary\eta_i\subset R$ and $\int{\eta}_i\cap R=\emptyset$. Take $z=R\cap G$ to be far from $\boundary\eta_i$ for each $i$.

Isotope and homotope $S_+$ and $S$ (respectively) near each Whitney disk $W_{i}$ as in Figure \ref{fig:crossediso} to be contained in a small neighborhood of $R\cup\eta_1\cup\cdots\cup\eta_n$. (Perform these moves in order of $i$. By dimensionality, $W_j$ is disjoint from tubes $\tilde{T}_k$ in the pictured support of $w_i$ for $i\le j$. In this support, the tubes $\tilde{T}_k$ that are not parallel to $R$ are parallel to $\eta_i$.) Call the resulting surfaces $\newtilde{S}_+$ and $\newtilde{S}$, respectively. 
The isotopy of $S_+$ to $\newtilde{S}_+$ takes tube $T_i$ to a tube $\newtilde{T}_i$. Now $\newtilde{S}_+$ is obtained from $\newtilde{S}$ by attaching tubes $\newtilde{T}_1,\ldots,\newtilde{T}_n$. The belt sphere of $T_i$ is carried to a curve $B_i$ on $\tilde{T}_i$. We call $B_i$ the belt sphere of $\tilde{T}_i$. Up to reparametrization of $\sigma_i$, $B_i$ bounds a disk perpendicular to $\tilde{S}_+$ which is centered at $b_i\subset\tilde{S}_+\cap R$. For $i\neq j$, take $b_i\neq b_j$.

\begin{figure}
\includegraphics[width=.5\textwidth]{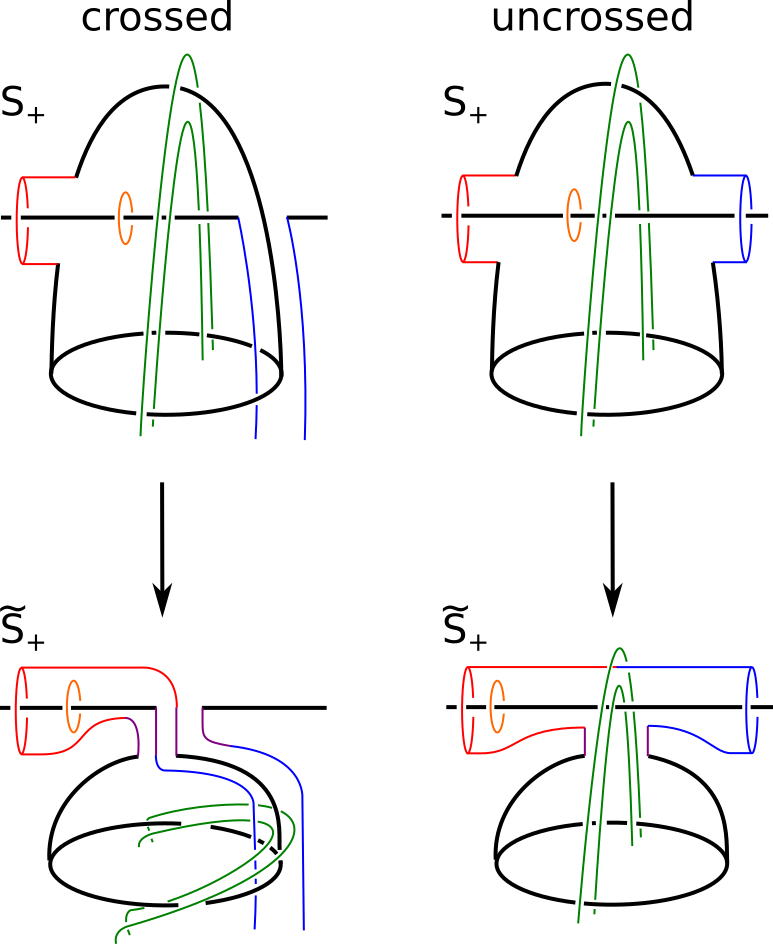}
\caption{{\bf{Top:}} the neighborhood of a crossed or uncrossed Whitney disk $W_i$. We draw $S_+\cap S$ in bold black, with the tubes $T_i$ in thin colored curves. In general, a tube $T_i$ may intersect this neighborhood many times. {\bf{Bottom:}} We isotope $S_+$ and homotope $S$ in a neighborhood of each $W_i$ to obtain $\newtilde{S}_+$ and $\newtilde{S}$, respectively. Now $\newtilde{S}_+\cap\newtilde{S}$ (in bold black) is contained in $R$.}
\label{fig:crossediso}
\end{figure}

For each $i=1,\ldots, n$, let $\alpha_i$ be an arc embedded in $R$ from $b_i$ to $z$. Take the arcs $\alpha_i,\alpha_j$ to be disjoint when $i\neq j$, and take $\alpha_i$ to be far from the endpoints of $\eta_1,\ldots,\eta_n$. Also take $\int{\alpha_i}\cap b_j=\emptyset$ for all $i$ and $j$.

 For $i=1,\ldots, n$, we now find a disk $\newtilde{C}_i$ whose boundary is the belt sphere $B_i$ of $\newtilde{T}_i$. See Figure \ref{fig:alpha} for an illustration of $\newtilde{C}_i$.

\begin{figure}
\includegraphics[width=\textwidth]{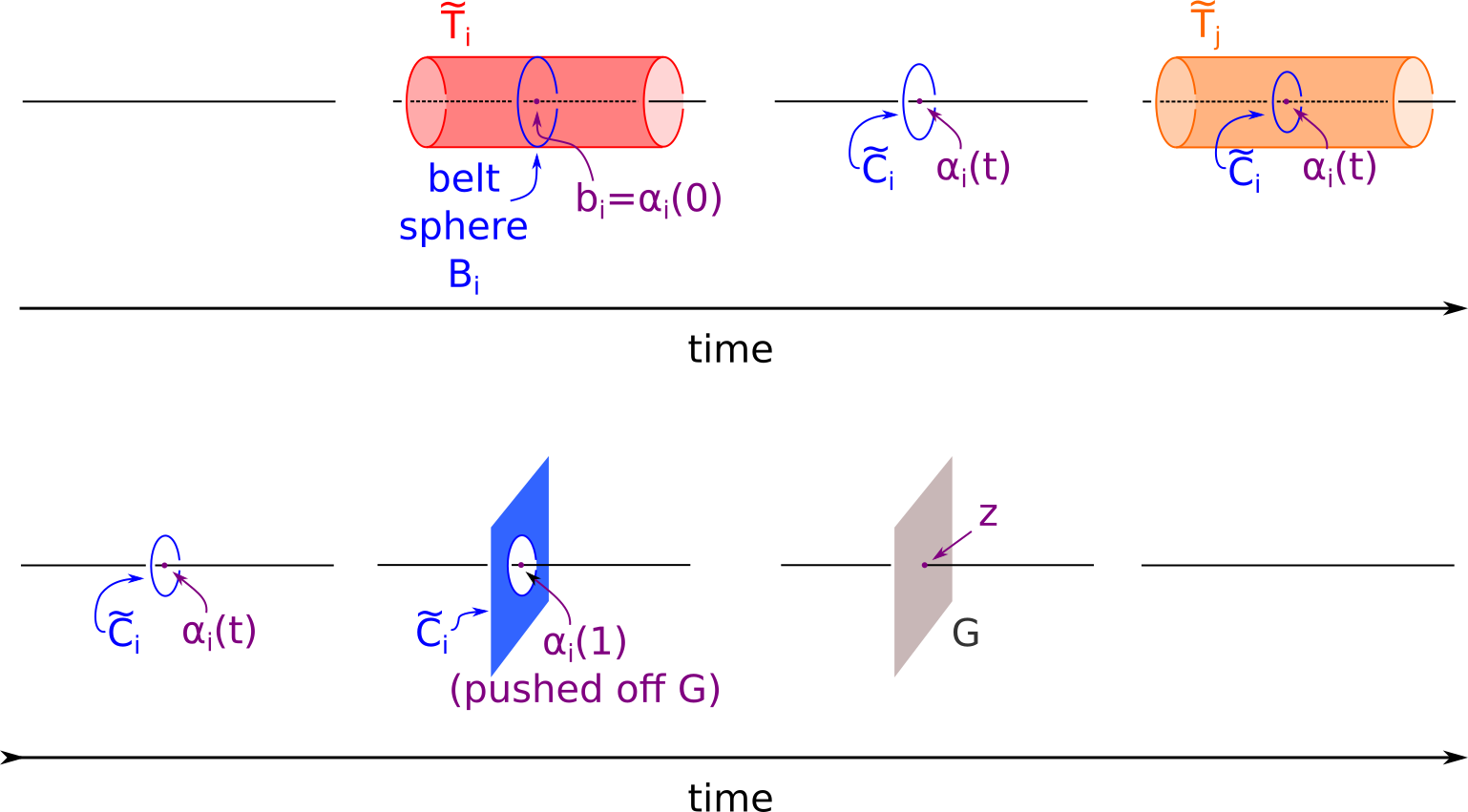}
\caption{The disk $C_i$ has boundary the belt sphere $B_i$ of $\newtilde{T}_i$, then follows the path of $\alpha_i(t)$ before being capped off by a disk in $G$. We push $C_1,\ldots, C_n$ off of $G$ and each other to obtain $\newtilde{C}_1,\ldots,\newtilde{C}_n$. Here, we draw a movie of $\newtilde{C}_i$. At each time slice, we draw a $3$-dimensional cross-section of $\nu(\alpha_i)$.}\label{fig:alpha}
\end{figure}

Let $T(\alpha_i)$ be a cylinder around $\alpha_i$, where $\alpha_i$ is framed so that $C(0)=B_i$ and $C(1)\subset G$. Let $C_i$ be the disk obtained by capping off $T(\alpha_i)$ with a disk in $G$ which does {\emph{not}} contain $z$. Take $T(\alpha_i)$ increasingly narrow so that $\int{C}_i$ does not intersect itself, $\newtilde{T}_j$ for any $j$, or $T(\alpha_k)$ for any $k\neq i$. The disks $C_1,\ldots, C_n$ all mutually intersect in a disk in $G$. Since $G$ has trivial normal bundle, we can push the disks $C_1,\ldots, C_n$ slightly off of $G$ in different directions to obtain disjoint disks $\newtilde{C}_1,\ldots,\newtilde{C}_n$ where $\boundary\newtilde{C}_i=\boundary C_i=B_i$. Note that the interior of $\newtilde{C}_i$ does not intersect $S_+$. 

%

Let $H'_i=\newtilde{C}_i\times I$, where the product direction is chosen so that $(\boundary\newtilde{C}_i)\times I\subset \newtilde{T}_i$. Let $R''$ be the sphere obtained from $S_+$ by compressing along each $\newtilde{C}_1,\ldots, \newtilde{C}_n$, so $\boundary(\newtilde{C}_i\times I)\subset S_+\cup R''$.

Let $\phi_s:X^4\to X^4\vert_{s\in [0,1]}$ be the ambient isotopy of $X^4$ taking $S_+$ to $\newtilde{S}_+$.

Let $M^3_2$ be a cobordism from $S_+$ to $R''$ in $X^4\times I$ given by \[M^3_2=\cup_{s\in[0,1]}\left(\phi_{s}(S_+)\times s/2\right)\sqcup_{i=1}^n (H'_i\times 1/2)\cup (R''\times[1/2,1]).\]
We illustrate $M^3_2$ in Figure \ref{fig:compresstubes}.

\begin{figure}
\includegraphics[width=\textwidth]{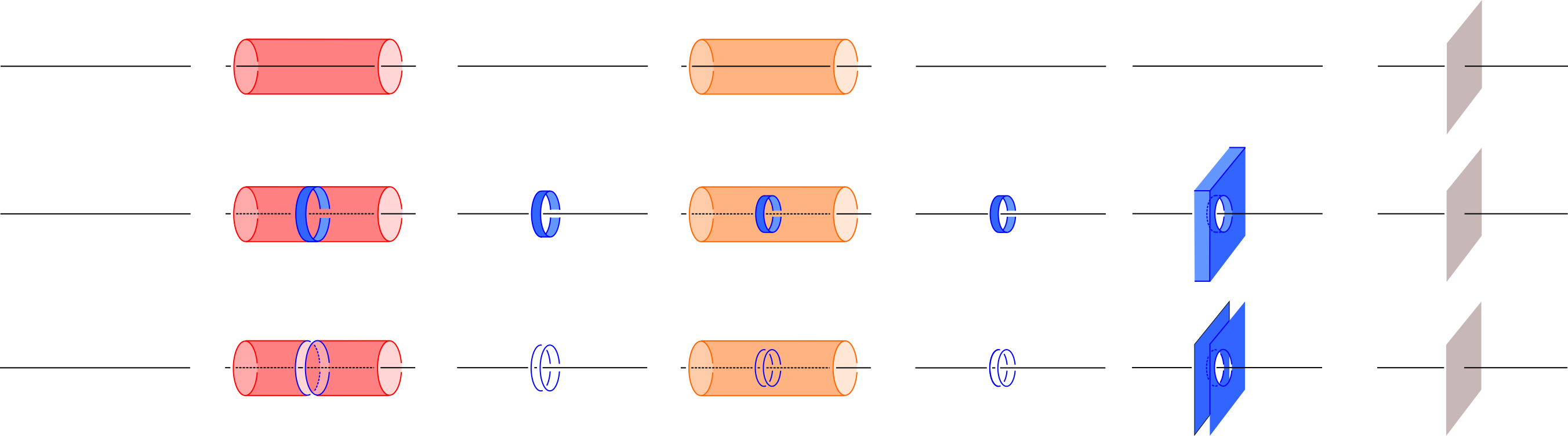}
\caption{The cobordism $M^3_2$ from $S_+$ to $R''$. {\bf{Top row:}} A movie of $\newtilde{S}_+$, as in Fig. \ref{fig:alpha}. This represents $M^3_2\cap(X^4\times (1/2-\epsilon))$. {\bf{Middle row:}} In blue, we draw the $2$-handle $H'_i$. The whole picture is $M^3_2\cap(X^4\times 1/2)$. {\bf{Bottom row:}} We compress $S_+$ along $\newtilde{C}_i$ to obtain $R''$. This is a movie of $M^3_2\cap(X^4\times (1/2+\epsilon))$.}\label{fig:compresstubes}
\end{figure}

Let $N^3$ be the cobordism from $R'$ to $R''$ in $X^4\times I$ obtained by concatenating $M^3_1$ and $M^3_2$. In words, $N^3$ is obtained from $M^3_1$ by attaching the $2$-handles $H'_1,\ldots, H'_n$. 
By Remark \ref{cancelremark}, these $2$-handles geometrically cancel the $1$-handles of $M^3_1$. Therefore, $N^3$ is a concordance from $R'$ to $R''$.


%
%
%
%

The sphere $R''$ intersects $G$ in exactly the point $z$. Now we will prove that $R''$ is isotopic to $R$, using the $4$-dimensional light bulb theorem.

%

\subsection{Proof that the concordance goes from $R'$ to $R$}\label{mainproof}

We will show that $R''$ is the realization of a tubed surface on $R$.




%

Recall that $S_+$ is obtained from $S$ by attaching tubes $T_1,\ldots, T_n$, where $T_i$ runs parallel to $\pi_S(\sigma_i)$ for an arc $\sigma_i$ between $x_j$ and $x_k$ (for some $j\neq k$), and that the isotopy $\phi_s$ from $S_+$ to $\newtilde{S}_+$ takes $T_i$ to tube $\newtilde{T}_i$. Let $\newtilde{B}_1,\ldots,\newtilde{B}_{2n}$  denote the components of $\newtilde{T}_1,\ldots,\newtilde{T}_n$ after compressing each $\newtilde{T}_i$ along $C_i$. Each $\newtilde{B}_i$ is a disk, and $R''$ is obtained from $\newtilde{S}$ by deleting disks bounded by $\boundary\newtilde{B}_i$ and then attaching $\newtilde{B}_i$ for each $i=1,\ldots, 2n$. (The disks deleted from $\newtilde{S}$ each meet a sheet of one of the $2n$ self-intersections of $\newtilde{S}$; the sphere $R''$ is embedded.)

Let $X_i=\phi_1($support of $w_i)$. See Figure \ref{fig:explainclean} (top three rows) for illustrations of $S$, $S_+$, and $R''$ in $X_i$.

%
%

For each $i=1,\ldots, n$, say $\pi_{S_i}^{-1}(W_i)=\boundary_i^1\cup\boundary_i^2$ where $\boundary_i^1$ and $\boundary_i^2$ are arcs in $A$. If $w_i$ is uncrossed, take $\boundary_i^2$ to have both endpoints in $L_y$. We perform the following operation to $R''$, illustrated in Figure \ref{fig:explainclean} (bottom). For $i=1,\ldots,n$:
\begin{itemize}
\item If $w_i$ is crossed, then take $R''\cap X_i$ as in Figure \ref{fig:explainclean} (third row, left).
\item If $w_i$ is uncrossed, suppose $\sigma_l$ crosses $\boundary_i^2$ for some $l$. (See Fig. \ref{fig:explainclean}, second row, rightmost.) Then some segment of $\newtilde{B}_r\cap X_i$ runs parallel to $\eta_i$, as in Figure \ref{fig:explainclean} (third row, third picture). For some $m\neq s$, $\newtilde{B}_m$ and $\newtilde{B}_s$ both have ends in $X_i$. Assume $r\neq s$ (by perhaps allowing $r=m$) and slide this segment of $\newtilde{B}_r$ over the disk $\newtilde{B}_s$ and out of $X_i$, as in Figure \ref{fig:explainclean} (third row, rightmost). Repeat for each intersection of a $\sigma_l$ curve (for any $l$) with $\boundary_i^2$.
\end{itemize}

Now we see that $R''$ is the realization of a tubed surface on $R$. The tube guide locus curves for $R''$ in $A$ are all of the form $\alpha_i,\beta_i$, and $\gamma_i$. Every $\newtilde{B}_j$ lies in a small neighborhood of $R$. Near the boundary of the disk $\newtilde{B}_j$, we find one of the two following situations:
\begin{itemize}
\item The ends of two $\newtilde{B}_j$'s join at a single tube parallel to $\eta_k$ where $w_k$ is uncrossed (recall from Remark \ref{fingerframed} that $\eta_k$ is a framed path).
\item The ends of two $\newtilde{B}_j$'s meet opposite ends of a double tube parallel to $\eta_k$ where $w_k$ is crossed.
\end{itemize}

Thus, the curves of the form $\tau_i$ for $R''$ are exactly $\{\eta_k\mid w_k$ uncrossed$\}$ while the curves of the form $\eta_i$ for $R''$ are exactly $\{\eta_k\mid w_k$ crossed.$\}$.

Assume $L$ is as in the hypothesis of Theorem \ref{maintheorem}. That is, each $2$-torsion element of $\pi_1(X^4,z)$ appears an even number of times in the multiset $\{[\eta_k]\mid w_k$ crossed$\}$. Then by Proposition \ref{tubeprop}, $R''$ is isotopic to $R$. Thus, $R'$ is concordant to $R$. This completes the proof of Theorem \ref{maintheorem} (and hence also Theorem \ref{easytheorem}). \qed

%
%
%
\begin{figure}
\includegraphics[width=\textwidth]{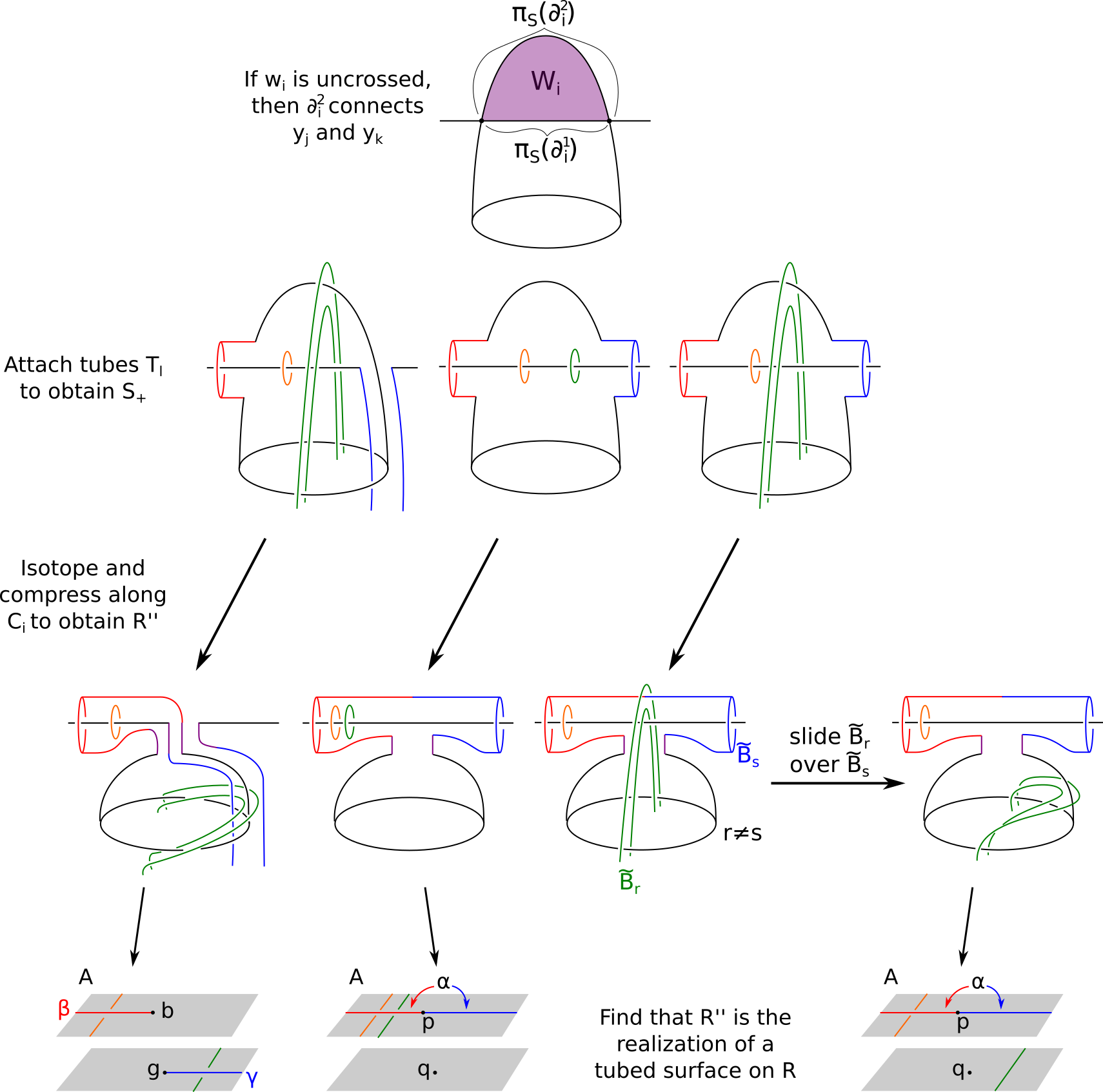}
\caption{{\bf{Top row:}} The Whitney disk $W_i$ with boundary in $S$. {\bf{Second row:}} we attach tubes $T_1,\ldots, T_n$ to $S$ to obtain $S_+$. In the leftmost picture, the Whitney move $w_i$ is crossed. In the middle and left pictures, $w_i$ is uncrossed. In the rightmost picture, some $\sigma_l$ intersects $\int{\boundary_i^2}$. (Here, the green tube is parallel to $\pi_S(\sigma_l))$. {\bf{Third row:}} We isotope $S_+$ to $\newtilde{S}_+$ and then compress along disks $C_i$ (not visible in this diagram) to obtain $R''$. In the third picture, some $\newtilde{B}_r$ corresponds to the previously pictured segment of $T_l$. There are two $\newtilde{B}_m,\newtilde{B}_s$ with ends pictured, with $m\neq s$. So without loss of generality, take $r\neq s$ and slide $\newtilde{B}_r$ over $\newtilde{B}_s$. {\bf{Bottom row:}} We find that $R''$ is the realization of a tubed surface on $R$. We give schematics for the tube guide locus arcs contained in $A$.}\label{fig:explainclean}
\end{figure}

\section{Concordance of surfaces of positive genus}

When $R$ and $R'$ are positive-genus surfaces rather than spheres, Gabai \cite{gabai} proves the following extension of the light bulb theorem.

\begin{theorem}[{\cite[Thm. 9.7]{gabai}}]\label{genus4dlight}
Let $X^4$ be an orientable $4$-manifold so that $\pi_1(X^4)$ has no $2$-torsion. Let $ R$ and $R'\subset X^4$ be orientable genus-$g$ surfaces embedded in $X^4$ so that $R$ and $R'$ have a mutual transverse sphere and $R'$ is homotopic to $R$. Moreover, assume the maps $\pi_1(R\setminus G)\to \pi_1(X^4\setminus G)$ and $\pi_1(R'\setminus G)\to \pi_1(X^4\setminus G)$ induced by inclusion are both trivial.

Then $R$ and $R'$ are isotopic.
\end{theorem}

Note that when $g=0$, Theorem \ref{genus4dlight} specializes to Theorem \ref{easy4dlight}.

The analogous extension of Theorem \ref{easytheorem} is thus as follows.

\begin{theorem}\label{genustheorem}
Let $X^4$ be an orientable $4$-manifold so that $\pi_1(X^4)$ has no $2$-torsion. Let $ R$ and $R'\subset X^4$ be orientable genus-$g$ surfaces embedded in $X^4$ so that $R$ has a mutual transverse sphere and $R'$ is homotopic to $R$. Moreover, assume the map $\pi_1(R\setminus G)\to \pi_1(X^4\setminus G)$ induced by inclusion is trivial.

Then $R$ and $R'$ are concordant.
\end{theorem}

\begin{proof}
%
%

By Theorem \ref{hirsch}, $R$ and $R'$ are regularly homotopic. We repeat the argument of Theorem \ref{easytheorem} to construct an embedded surface $R''$ which is the realization of a tubed surface on $R$ so that $R''$ is concordant to $R'$ and $R\cap G=\pt$. (In exactly the same fashion as in Theorem \ref{easytheorem}, we build a concordance from $R$ to $R''$ by attaching a $1$-handle to $R\times I$ for each finger move in the regular homotopy, and then attach cancelling $2$-handles using the transverse sphere $G$.)

Note that $R''$ is built from $R'$ by surgery along immersed $3$-balls ($1$- and $2$-handle pairs) which meet $R'$ in a disk and can thus be homotoped to be trivial (see Figure \ref{fig:homotopy} for an illustration). Therefore, $R''$ is homotopic to $R'$ and hence $R$. Moreover, every loop in $R''$ can be isotoped off the tubes attached to $R$ and into $R$ itself, so the map $\pi_1(R''\setminus G)\to \pi_1(X^4\setminus G)$ induced by inclusion is trivial. Then by Theorem \ref{genus4dlight}, $R''$ is isotopic to $R$, so $R'$ is concordant to $R$.

\begin{figure}
\includegraphics[width=\textwidth]{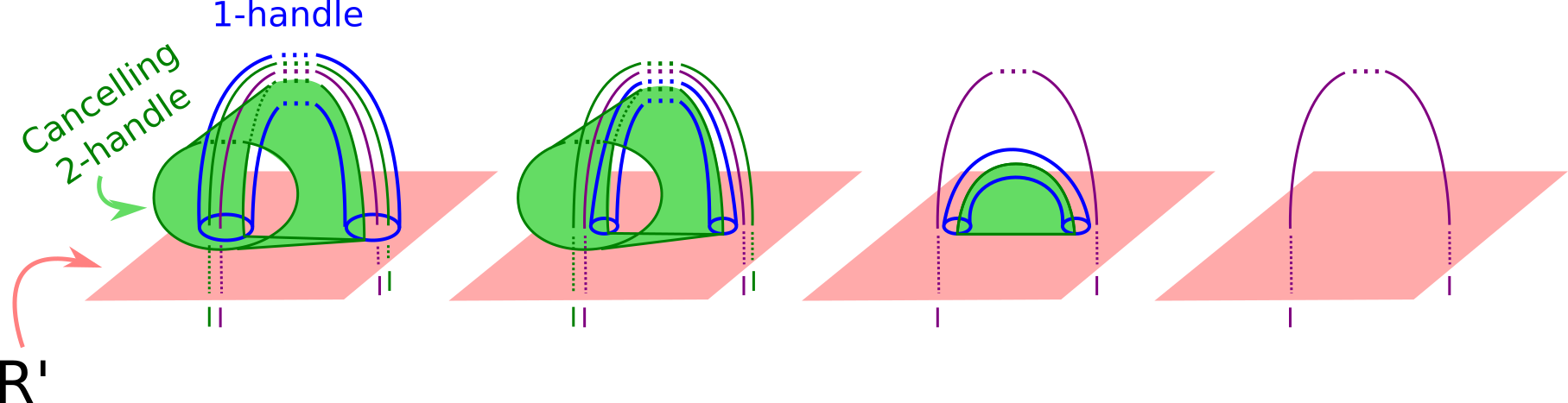}
\caption{Left to right: Part of a homotopy from $R''$ to $R$. Here we draw $R'$ and a schematic of a $1$-handle and the (collar of a) core of the $2$-handle which geometrically cancels it, projected to one $X^4\times t$. The cancelling $2$-handle and other $2$-handle may intersect the $1$-handle. The $1$-handle and cancelling $2$-handle together form an immersed $3$-ball, which we shrink over time during the homotopy. To obtain $R'$, we repeat for each $1$-handle of $M^3_1$.}\label{fig:homotopy}
\end{figure}

\end{proof}


\begin{thebibliography}{9}

\bibitem[DNPR]{davis}{C. W. Davis, M. Nagel, J. Park, and A. Ray, {\emph{Concordance of knots in $S^1\times S^2$}}, arXiv:1707.04542 [math.GT] Jul. 2017 (to appear in J. London Math. Soc.).}
\bibitem[FQ]{quinn}{M. H. Freedman and F. Quinn, Topology of 4-Manifolds, Princeton Math. Ser., {\bf{39}}, Princeton Univ. Press, Princeton, NJ, 1990.}
\bibitem[G]{gabai}{D. Gabai, {\emph{The 4-dimensional light bulb theorem}}, arXiv:1705.09989 [math.GT] May 2017.}
\bibitem[H]{hirsch}{M. W. Hirsch, {\emph{Immersions of manifolds}}, Trans. Amer. Math. Soc. {\bf{93}}:2 (1959), 242--276.}
\bibitem[M]{otherpaper}{M. Miller, {\emph{Concordances from the standard surface in $S^2\times S^2$}}, arXiv:1709.07176 [math.GT] Sep. 2017.}
\bibitem[Sc]{hannah}{H. R. Schwartz, {\emph{Equivalent non-isotopic spheres in 4-manifolds}}, arXiv:1806.07541 [math.GT] Jun. 2018.}
\bibitem[Sm]{smale}{S. Smale, {\emph{A classification of immersions of the two-sphere}}, Trans. AMS (1957), 281--290.}
\bibitem[Su]{sunukjian}{N. Sunukjian, {\emph{Surfaces in 4-manifolds: concordance, isotopy, and surgery}}, Int. Math. Res. Notices (2015):17, 7950--7978.}
\bibitem[Y]{yildiz}{E. Z. Yildiz, {\emph{A note on knot concordance}}, Algebra. Geom. Topol. {\bf{18}} (2018), 3119--3128.}
\end{thebibliography}
\end{document}